\documentclass[12pt]{amsart}
\usepackage{lmodern,graphicx}
\usepackage{latexsym}
\usepackage{amsfonts}
\usepackage{pgfplots}
\usepackage{tikz}
\usetikzlibrary{arrows,shapes,positioning}
\usetikzlibrary{decorations.markings}
\tikzstyle arrowstyle=[scale=1]
\tikzstyle directed=[postaction={decorate,decoration={markings,
    mark=at position .65 with {\arrow[arrowstyle]{stealth}}}}]
\tikzstyle reverse directed=[postaction={decorate,decoration={markings,
    mark=at position .45 with {\arrowreversed[arrowstyle]{stealth};}}}]
\usetikzlibrary{backgrounds}
\usepackage[german,english]{babel}

\usepackage{amsmath,amsthm,amssymb}

\newtheorem{question}{Question}
\newtheorem{proposition}{Proposition}[section]
\newtheorem{theorem}[proposition]{Theorem}

\newtheorem{thm}{Theorem}[section]
\newtheorem{lemma}[thm]{Lemma}
\newtheorem{corollary}[thm]{Corollary}
\newtheorem{conjecture}[thm]{Conjecture}

\begin{document}

\bigskip

\begin{center}
\Large
Andrews-Curtis and Nielsen equivalence relations on some infinite groups
\normalsize

\bigskip

Aglaia Myropolska\footnote{The author acknowledges the support of the Swiss National Science Foundation, grant 200021\_144323.}

\end{center}

\bigskip

The Andrews-Curtis conjecture asserts that, for a free group $F_n$ of rank $n$ and a free basis $(x_1,...,x_n)$, any normally generating tuple $(y_1,...,y_n)$ is Andrews-Curtis equivalent to $(x_1,...,x_n)$. This equivalence corresponds to the actions of $\operatorname{Aut}F_n$ and of $F_n$ on normally generating $n$-tuples. The equivalence corresponding to the action of $\operatorname{Aut}F_n$ on generating $n$-tuples is called Nielsen equivalence. The conjecture for arbitrary finitely generated group has its own importance to analyse potential counter-examples to the original conjecture. We study the Andrews-Curtis and Nielsen equivalence in the class of finitely generated groups for which every maximal subgroup is normal, including nilpotent groups and Grigorchuk groups.

\section{Introduction}
The famous Andrews-Curtis conjecture \cite{AC,Lub} can be stated as follows:
\begin{conjecture}[The Andrews-Curtis conjecture] Let $F_n$ be a free group of rank $n\geq 2$. If  $\{x_1,...,x_n\}$ is a free basis and $\{r_1,...,r_n\}$ is a normally generating set of $F_n$, then $(r_1,...,r_n)$ and $(x_1,...,x_n)$ are Andrews-Curtis equivalent.
\end{conjecture}
Let us give some background related to this conjecture. Let $G$ be a finitely generated group. \emph{The rank $\operatorname{rank}(G)$} of a group $G$ is the minimal number of generators of $G$. The following transformations of the set $G^n, n\geq 1$, are called \emph{elementary Nielsen moves}:
\begin{align*}
R_{ij}^{\pm}(x_1,...,x_i,..., x_j,..., x_n)&=(x_1,...,x_{i}x_j^{\pm 1},...,x_j,..., x_n),\\
L^{\pm}_{ij}(x_1,...,x_i,..., x_j,..., x_n)&=(x_1,...,x_j^{\pm 1}x_{i},...,x_j,..., x_n),\\
 I_j(x_1,...,x_j,...,x_n)&=(x_1,...,x_j^{-1},...,x_n),\\
\end{align*}
where $1\leq i,j \leq n$, $i\neq j$.
Elementary Nielsen moves transform generating sets of $G$ into generating sets. Two generating sets are called \emph{Nielsen equivalent} if one is obtained from the other by a finite chain of elementary Nielsen moves. 

Elementary Nielsen moves together with the transformations
\begin{align*} AC_{i,s}(x_1,...,x_i,...,x_n)=(x_1,...,x_i^{s},...,x_n)\end{align*} where $1\leq i\leq n,\text{ }s\in S\cup S^{-1}\subset G $, $S$ is a fixed subset of $G$, form the set of \emph{elementary Andrews-Curtis moves relative to $S$}, or shortly AC$_{S}$-moves. Elementary AC$_S$-moves transform \emph{normally generating sets} (sets which generate $G$ as a normal subgroup) into normally generating sets. Two normally generating sets are called \emph{Andrews-Curtis equivalent relative to $S$} (AC$_S$ equivalent) if one is obtained from the other by a finite chain of elementary AC$_S$-moves. In the case when $S=G$, we refer to AC-moves.

As in \cite{Kerv}, the\emph{ weight $w(G)$} of a group $G$ is the minimal number of normal generators of $G$.

Given an integer $n\geq w(G)$ and a subset $S\subset G$, the \emph{Andrews-Curtis graph} (AC-graph) $\Delta_{n,S}(G)$ of the group $G$ is defined as follows: 
\begin{itemize}
\item[-] the set of vertices consists of normally generating $n$-tuples, i.e. \begin{equation*} V_{\Delta_{n,S}}(G)=\{(g_1,...,g_n) \in G^n\mid\ll g_1,...,g_n \gg  = G\};\end{equation*}
\item[-] two vertices are connected by an edge if one of them is obtained from the other by an elementary $AC_S$-move. 
\end{itemize}
If $S=G$ we simplify the notation to $\Delta_n(G)$. Clearly if $S$ generates $G$, the graph $\Delta_{n,S}(G)$ is connected if and only if the graph $\Delta_n(G)$ is connected. The graph $\Delta_{n,S}(G)$ is a regular graph which admits loops and multiple edges. 

The Andrews-Curtis conjecture can be reformulated in terms of graphs:
\begin{conjecture}[The Andrews-Curtis conjecture] For a free group $F_n$ of rank $n\geq 2$, the Andrews-Curtis graph $\Delta_n(F_n)$ is connected.\end{conjecture}
\bigskip

Let $G$ be a finitely generated group and $n\geq \operatorname{rank}(G)$. The set of $n$-tuples $G^n$ can be identified with the set of homomorphisms from $F_n$ to $G$ and, therefore, the set of generating $n$-tuples can be identified with the set of epimorphisms $\operatorname{Epi}(F_n,G)$. Hence there are natural actions of the automorphism group $\operatorname{Aut}F_n$ on both $G^n$ and $\operatorname{Epi}(F_n,G)$, by precomposition. Elementary Nielsen moves can be seen as elements of $\operatorname{Aut}F_n$; moreover, by a result of Nielsen (see \cite{LS}, Chap.\ I, Prop.\ 4.1), they generate $\operatorname{Aut}F_n$. Therefore the orbits of the action of $\operatorname{Aut}F_n$ on $\operatorname{Epi}(F_n,G)$ are the \emph{Nielsen equivalence classes}. Transitivity of this action has been studied in different contexts, see \cite{Ev06,Lub,Pak}.

We define the \emph{Nielsen graph}\footnote[3]{Also called the Extended Product Replacement Graph.}  $\Gamma_n(G)$, $n\geq \operatorname{rank}(G)$, as follows: \begin{itemize} \item[-] the set of vertices consists of generating $n$-tuples, i.e. \begin{equation*}V_{\Gamma_n}(G)=\{(g_1,...,g_n)\in G^n\mid \langle g_1,...,g_n \rangle = G\};\end{equation*}
\item[-]  two vertices are connected by an edge if one of them is obtained from the other by an elementary Nielsen move.
\end{itemize}
Observe that the graph $\Gamma_n(G)$ is connected if and only if the action of $\operatorname{Aut}F_n$ on $\operatorname{Epi}(F_n,G)$ is transitive.

As an example, let us consider  a finitely generated abelian group $G$. Then a normal generating set is just a generating set and elementary Andrews-Curtis moves coincide with elementary Nielsen moves. Therefore the two graphs $\Delta_n(G)$ and $\Gamma_n(G)$ coincide. Moreover, we have the following:
\begin{theorem} [\cite{NN, DiacGrah,Oan}]
Let $G$ be a finitely generated abelian group given as \begin{equation*}G\cong \mathbb{Z}_{m_1} \times \mathbb{Z}_{m_2}\times...\times \mathbb{Z}_{m_r}\times \mathbb{Z}^s\end{equation*}
where $r,s\geq 0$, $m_1,...,m_r\geq 2$ and $m_1|m_2|...|m_r$. Then $\operatorname{rank}(G)=r+s$ and 

\begin{itemize}
\item $\Gamma_n(G)$ is connected  for $n\geq r+s+1$;
\item if $r=0$, i.\ e., $G\cong \mathbb{Z}^s$, then $\Gamma_s(G)$ is connected;
\item otherwise if $m_1=2$ then $\Gamma_{r+s}(G)$ is connected and if $m_1\neq 2$ then $\Gamma_{r+s}(G)$ has $\frac{\phi(m_1)}{2}$ connected components, where $\phi(m)$ is the Euler function (the number of integers less than $m$ which are coprime with $m$).
\end{itemize}
\label{thm:diakgrah}
\end{theorem} 
Along with this result there are several partial results about the connectedness of $\Gamma_n(G)$  for particular families of groups. For instance, $\Gamma_{n+1}(G)$ is connected when $G$ is a finitely generated nilpotent group of $\operatorname{rank}$ $n$ \cite{Evans}. For further known results we refer the reader to Section \ref{N} on Nielsen equivalence, and to the references therein.

Let us get back to the Andrews-Curtis conjecture, which is still open. There have been several attempts to construct counter-examples. See Section \ref{AC} on Andrews-Curtis equivalence for more on that. A possible way to disprove the conjecture would be to find two normally generating systems of $F_n$ such that their images in some finitely generated group  are not Andrews-Curtis equivalent. This motivates to analyse the connected components of the Andrews-Curtis graph of  finitely generated groups. One of the few positive results in this direction is that $\Delta_n(G)$ is connected when $G$ is a free soluble group of rank $n$ \cite{Myas}, and its proof can be adapted for free nilpotent groups.

\medskip
Borovik, Lubotzky and Myasnikov \cite{BorLubMyas} studied connected components of $\Delta_n(G)$ for finite groups. In particular, they proved the following:
\begin{theorem}[\cite{BorLubMyas}]
Let $G$ be a finite group and $n\geq \operatorname{max}\{w(G),2\}$. Then two normally generating tuples $U,V$ are AC equivalent if and only if they are AC equivalent in the abelianization $\operatorname{Ab}(G)=G/[G,G]$. In other words, the connected components of the AC-graph $\Delta_n(G)$ are precisely the preimages of the connected components of the AC-graph $\Delta_n(\operatorname{Ab}(G))$. \label{blm}
\end{theorem}

Furthermore, they raised the question whether such a criterion holds for the Grigorchuk group \cite{ Grig80,GrigSelecta}: this is a 3-generated residually finite 2-group which is just-infinite; hence Theorem \ref{blm} holds in every proper quotient of the group. It equally makes sense to ask the question for other just-infinite groups. We will now discuss a class of groups which includes the Grigorchuk group (see e.g. \cite{dlHar}, Chap. VIII for more on the Grigorchuk group).

\medskip
In this paper we focus on the class \emph{$\mathfrak{C}$ of finitely generated groups for which every maximal subgroup is normal}. 

A finite group is in $\mathfrak{C}$ if and only if it is nilpotent \cite[5.2.4]{Rob}. Moreover $\mathfrak{C}$ contains finitely generated nilpotent groups, because any maximal subgroup of a nilpotent group is normal \cite[I.70]{Bour}. The class $\mathfrak{C}$ also contains the family of Grigorchuk groups $(G_{\omega})_{w\in \Omega}$ \cite{Grig80, Gr84} indexed by sequence in $\Omega=\{0,1,2\}^{\mathbb{N}}$, and Gupta-Sidki $p$-groups \cite{Pervova, Per5}.

We will now state our first result: 

\begin{theorem}\label{GACC}
Let $G$ be in $\mathfrak{C}$ and $n\geq w(G)$. Then two normally generating $n$-tuples $U,V$ are AC equivalent if and only if they are AC equivalent in the abelianization $\operatorname{Ab}(G)=G/[G,G]$. In other words, the connected components of the AC-graph $\Delta_n(G)$ are precisely the preimages of the connected components of the AC-graph $\Delta_n(\operatorname{Ab}(G))$.
\end{theorem}

Together with Theorem \ref{thm:diakgrah} our result describes the connected components of the Andrews-Curtis graph for  groups in $\mathfrak{C}$. 
\begin{corollary}
For the Grigorchuk group $G$ the Andrews-Curtis graph $\Delta_n(G), n\geq 3$, is connected.
\label{ACGrig}
\end{corollary}

It is not hard to see that, for groups in $\mathfrak{C}$, the set of vertices of $\Delta_n(G)$ and $\Gamma_n(G)$ coincide (see Proposition \ref{obs}). Therefore connectedness of $\Gamma_n(G)$ implies connectedness of $\Delta_n(G)$. Our other results are devoted to connectedness of $\Gamma_n(G)$. We have already mentioned that $\Gamma_n(G)$ is connected for nilpotent groups when $n\geq \operatorname{rank}(G)+1$. Here we study connectedness of $\Gamma_n(G)$ for nilpotent groups when $n=\operatorname{rank}(G)$. For free nilpotent groups, as a direct corollary of the result by Andreadakis \cite{Andreadakis} and Bachmuth \cite{Bach}, we obtain the following: 
\begin{theorem}[Section $4.2$]
Let $G$ be the free nilpotent group of rank $n\geq 2$ and nilpotency class $c$. Then
$\Gamma_n(G)$ is connected if and only if $c=1$ or $2$. 
\label{thm1.4}
\end{theorem}
Furthermore, let $\mathcal{H}_k$ be the discrete Heisenberg group of rank $2k$ (Section \ref{NilHeis}).
\begin{theorem}[Section $4.2$]
The graph $\Gamma_n(\mathcal{H}_k)$ is connected when $n\geq 2k$.
\label{thm1.5}
\end{theorem}

In the end we prove the connectedness of $\Gamma_n(G)$ for $p$-groups in $\mathfrak{C}$ with $n\geq \operatorname{rank}(G)+1$:

\begin{theorem}[Section $4.1$]
Let $G$ be a $p$-group in $\mathfrak{C}$. Then $\Gamma_n(G)$ is connected when $n\geq \operatorname{rank} (G)+1$.
\label{thm1.3}
\end{theorem}

We remind that the Gupta-Sidki $p$-group, defined for every odd prime $p$, is $2$-generated.
\begin{corollary}
 $\Gamma_n(G)$ is connected for the Grigorchuk group and the Gupta-Sidki group, when $n\geq 4$ and $n\geq 3$ respectively.\footnote{Connectedness of $\Gamma_n(G)$, $n\geq 4$, for the Grigorchuk group was also obtained in \cite{MalPak}.}
\end{corollary}

\begin{corollary}
Let $G$ be the Gupta-Sidki $p$-group, $p>3$. Then $\Gamma_2(G)$ has at least $\frac{p-1}{2}$ connected components.
\label{gs}
\end{corollary}

\begin{question} Is $\Gamma_3(G)$ connected for the Grigorchuk group? Is $\Gamma_2(G)$ connected for the Gupta-Sidki $3$-group?
\end{question}
\bigskip

\textbf{Acknowledgements.} 
It is a pleasure to thank my advisor Tatiana Smirnova-Nagnibeda and Pierre de la Harpe for numerous discussions, valuable questions and beneficial comments, and also for revising the introduction. I am also grateful to Tsachik Gelander and Rostislav Grigorchuk for discussions, Alexei Myasnikov for bringing my attention to the problem in \cite{BorLubMyas}, and Christian Hagendorf for comments on a preliminary version of the manuscript.

\section{Class $\mathfrak{C}$}\label{AC}

The class $\mathfrak{C}$ is a class of finitely generated groups for which every maximal subgroup is normal.

We recall that the \emph{Frattini subgroup} $\Phi(G)$ of a group $G$ is the intersection of all maximal subgroups of $G$, and $\Phi(G)=G$ if $G$ does not have maximal subgroups.

\begin{proposition} Let $G$ be a finitely generated group. Then $G$ is in $\mathfrak{C}$ if and only if all normally generating sets are generating sets. Moreover, if $G$ is in $\mathfrak{C}$, then $[G,G]\leq \Phi(G)$.
\label{obs}
\end{proposition}
\begin{proof} Let $G$ be in $\mathfrak{C}$ and assume by contradiction that there is a normally generating set $S$ which is not a generatng set. Since $G$ is finitely generated, any proper subgroup is contained in some proper maximal subgroup \cite{Neum}, therefore $\langle S\rangle\leq M<G$ for some proper maximal subgroup $M$ . Then $\ll S\gg=G\leq M^G=M< G$. It is a contradiction.

Conversely, we will prove that if all normally generating sets are generating sets then all maximal subgroups are normal. We prove the equivalent statement: if there is a maximal subgroup $M$ which is not normal, then there exists a normally generating set which is not a generating set. We take as a normally generating set $S=M$. Then $\ll S\gg=G$ since $M$ is maximal and $\langle S\rangle\neq G$ since $M$ is proper. 

Futhermore, if $G$ is in $\mathfrak{C}$ then for any maximal subgroup $M$ the quotient $G/M$ is isomorphic to $\mathbb{Z}/p\mathbb{Z}$ by the correspondence theorem, in particular, $G/M$ is abelian. Therefore $[G,G]\leq M$ and we conclude that $[G,G]\leq \Phi(G)$.\end{proof}

\begin{proposition} Let $G$ be in $\mathfrak{C}$ and $n\geq 1$. Then the natural maps 
\begin{align*} &\pi_n: Epi(F_n,G)\rightarrow Epi(F_n, G/[G,G]), \\ &\phi_n: Epi(F_n,G)\rightarrow Epi(F_n, G/\Phi(G)) \end{align*} are surjective.
\label{lemma:grig}
\end{proposition}
\begin{proof} We will now prove the case of $\pi_n$. Consider the projection $\pi: G\rightarrow G/[G,G]$. 

If $n<\operatorname{rank}(G/[G,G])$, the sets $\operatorname{Epi}(F_n,G/[G,G])$ and $\operatorname{Epi}(F_n, G)$ are empty and there is nothing to prove. Assume now that $n\geq \operatorname{rank}(G/[G,G])$. Let $s_1, ..., s_n \in G$ such that $\pi(s_1),...,\pi(s_n)$ generate $G/[G,G]$. Suppose $\langle s_1,...,s_n\rangle \leq M$ for some maximal subgroup $M$ of $G$; we will obtain a contradiction. By Proposition \ref{obs}, $[G,G]\leq M$, and we have $M/[G,G]=\pi(M)=\pi(G)=G/[G,G]$. Hence $\{e\}=\frac{G/[G,G]}{M/[G,G]}\cong G/M$. This is a contradiction since $M$ is a proper subgroup of $G$. The proof for $\phi_n$ repeats the one above with $[G,G]$ replaced by $\Phi(G)$. \end{proof}

\section{Andrews-Curtis equivalence}\label{AC}
There are doubts as to whether the Andrews-Curtis conjecture is true. A possible way to disprove it would be to find two normally generating systems of $F_n$ such that their images are not Andrews-Curtis equivalent in some finitely generated group. 

Akbulut and Kirby \cite{AkKir} suggest a series of potential counter-examples for $F_2=\langle x,y\rangle$, i.e., normally generating tuples which are not known to be AC equivalent to $(x,y)$: \begin{equation}(u,v_l)=(xyxy^{-1}x^{-1}y^{-1},x^l y^{-(l+1)}),\text{ } l\geq 1. \label{ce}\end{equation}


It was suggested in \cite{BorKhukMyas} that one could confirm one of these potential counter-examples by showing that for some homomorphism $\phi: F_2\rightarrow G$ to a finite group $G$, the images of the pairs (\ref{ce}) are not Andews-Curtis equivalent. 

Notice that in an abelian group of rank $2$: \begin{equation*}(xyxy^{-1}x^{-1}y^{-1},x^l y^{-(l+1)})\sim_{AC} (x,y)\end{equation*} so for every homomorphism $\phi:F_2\rightarrow A$ into an abelian group $A$, the images of the pairs (\ref{ce}) are AC equivalent. 

In view of the latter, \cite{BorLubMyas} considered the class of groups with the following property: for any $n\geq \operatorname{max}\{w(G),2\}$, two normally generating $n$-tuples $U,V$ are AC equivalent in $G$ if and only if their images are AC equivalent in the abelianization $\operatorname{Ab}(G)$. Therefore groups from this class will not confirm the potential counter-examples (\ref{ce}). Borovik, Lubotzky and Myasnikov \cite{BorLubMyas} proved that all finite groups belong to this class. They also ask whether it is true for the Grigorchuk group. Our Theorem \ref{GACC} answers their question positively.

\bigskip
We will now proceed to the proof of Theorem \ref{GACC}.
\\First, we present a few well-known properties of the Frattini subgroup:
\begin{lemma}
Let $G$ be a group. 
\begin{enumerate}
\item \cite[5.2.12]{Rob}  Frattini subgroup $\Phi(G)$ is equal to the set of non-generators of $G$, i.e., if $g\in \Phi(G)$ and $\langle g,X\rangle=G$ then $\langle X\rangle=G$. 
\item \cite{Evans} Let $G=\langle x_1,...,x_n\rangle$ and $\varphi_1,...,\varphi_n \in \Phi(G)$. Then \begin{center}$\langle x_1\varphi_1,...,x_n\varphi_n\rangle=G$.\end{center}
\end{enumerate}
\end{lemma}

\begin{proof}[Proof of Theorem \ref{GACC}] Let $G$ be in $\mathfrak{C}$. Then by Proposition \ref{obs} normally generating sets coincide with generating sets of $G$ and therefore $\operatorname{rank}(G)=w(G)$. 

Consider two generating tuples $U,V$ in $G$ which are AC equivalent. Then for any normal subgroup $N\lhd G$ their images in $G/N$ are AC equivalent. In particular, it is true for $N=[G,G]$. 

Assume now that the images of $U,V$ in $\operatorname{Ab}(G)$ are AC equivalent. Let us first consider the case $n=\operatorname{rank}(G)$ and $\operatorname{rank}(G)=2$. Proving that $U=(u_1,u_2),V=(x,y)$ are AC equivalent is equivalent to proving that $(x,y)$ and $(x\varphi_1, y\varphi_2)$ are AC equivalent for $\forall \varphi_1,\varphi_2 \in [G,G]$. 

Denote by $\operatorname{ord}_G(g)$ the order of $g$ in $G$. 
One computes that 
\begin{equation*}\{x^{n_1}y^{n_2}| n_1\in (-\operatorname{ord}_G(x),\operatorname{ord}_G(x)), n_2\in (-\operatorname{ord}_G(y),\operatorname{ord}_G(y))\}\end{equation*} is a right coset representative system for $G$ mod $[G,G]$.  Using the Reidemeister-Schreier rewriting process we find a set $S$ of generators for $[G,G]$: \begin{center} $S=\{x^{n_1}y^{n_2}x(x^{n_1+1}y^{n_2})^{-1}, n_1\in (-\operatorname{ord}_G(x),\operatorname{ord}_G(x)), n_2\in (-\operatorname{ord}_G(y),\operatorname{ord}_G(y))\}$. \end{center} 

We will proceed by induction on the length of $\varphi_1\in [G,G]$ in generators of $S$ to prove that $(x,y)$ and $(x\varphi_1, y)$ are AC equivalent. Let $\varphi_1=s\varphi_1'$ with $s\in S$ and $\varphi_1' \in [G,G]$ such that $l_S(\varphi_1')<l_S(\varphi_1)$. Then 
\begin{center}
$(x\varphi_1,y)=(xs\varphi_1',y)=(x\cdot x^{n_1}y^{n_2}xy^{-n_2}x^{-n_1-1}\varphi_1',y)\sim_{I_2,AC} (x\cdot x^{n_1}y^{n_2}xy^{-n_2}x^{-n_1-1}\varphi_1',x^{n_1+1}y^{-1}x^{-n_1-1})\sim_{L_{12} \text{ } |n_1|+1 \text{ times}} (x^{n_1+2}y^{-n_2}x^{-n_1-1}\varphi_1',x^{n_1+1}y^{-1}x^{-n_1-1})\sim_{AC, I_2}(x^{n_1+2}y^{-n_2}x^{-n_1-1}\varphi_1',y)\sim_{AC}(x^{n_1+2}y^{-n_2}x^{-n_1-1}\varphi_1',x^{n_1+2}yx^{-n_1-2})\sim_{L_{12} \text{ }|n_2| \text{ times}} (x\varphi_1',x^{n_1+2}yx^{-n_1-2})\sim_{AC} (x\varphi_1',y)$. 
\end{center} 
By induction we conclude that $(x\varphi_1,y)$ and $(x,y)$ are AC equivalent. 

\bigskip
Now let $\hat{x}=x\varphi_1$ and $G=\langle \hat{x},y\rangle$. By a similar procedure using the Reidemeister-Schreier rewriting process we find $[G,G]=\langle S\rangle$ where \begin{equation*}S=\{\hat{x}^{n_1}y^{n_2}\hat{x}(\hat{x}^{n_1+1}y^{n_2})^{-1}, n_1\in (-\operatorname{ord}_G(\hat{x}),\operatorname{ord}_G(\hat{x})), n_2\in (-\operatorname{ord}_G(y),\operatorname{ord}_G(y))\}.\end{equation*} Let $\varphi_2=s\varphi_2'$ with $s\in S$ and $\varphi_2' \in [G,G]$ such that $l_S(\varphi_2')<l_S(\varphi_2)$. Then 
\begin{center}
$(\hat{x},y\varphi_2)=(\hat{x},y\hat{x}^{n_1}y^{n_2}\hat{x}(\hat{x}^{n_1+1}y^{n_2})^{-1}\varphi_2')\sim_{I_1,AC}(y\hat{x}^{n_1}y^{n_2}\hat{x}^{-1}y^{-n_2}\hat{x}^{-n_1}y^{-1},y\hat{x}^{n_1}y^{n_2}\hat{x}(\hat{x}^{n_1+1}y^{n_2})^{-1}\varphi_2')\sim_{L_{21},AC,I_1} (\hat{x},y\hat{x}^{-1}\varphi_2')\sim_{AC}(y\hat{x}y^{-1},y\hat{x}^{-1}\varphi_2')\sim_{L_{21},AC} (\hat{x},y\varphi_2')$.\end{center} By induction $(\hat{x},y)$ and $(\hat{x}, y\varphi_2)$ are AC equivalent. We conclude that $(x,y)$ and $(x\varphi_1,y\varphi_2)$ are AC equivalent. 

\medskip
Let $n=\operatorname{rank}(G)\geq 3$ and $G=\langle x_1,...,x_n \rangle$. Then $[G,G]=\langle S\rangle$, with $S=\{x_1^{m_1}...x_n^{m_n}x_l(x_1^{m_1}...x_l^{m_l+1}...x_n^{m_n})^{-1}, m_i\in (-\operatorname{ord}_G(x_i), \operatorname{ord}_G(x_i)),$\\$ m_n\neq 0, 1\leq l<n\}$. Then the proof above can be repeated  with more similar cases to consider. 

\medskip
Finally, assume $n>\operatorname{rank}(G)$, and for simplicity $\operatorname{rank}(G)=2$. Let us fix a system of generators $\{x,y\}$ of $G$. 

Since $G/[G,G]$ is abelian and $n>\operatorname{rank}(G/[G,G])$, the $\Delta_n(G/[G,G])=\Gamma_n(G/[G,G])$ is connected by Theorem \ref{thm:diakgrah}. We need to prove that $\Delta_n(G)$ is connected. This is equivalent to proving that $(x\varphi_1, y\varphi_2,\varphi_3,...,\varphi_n)\sim_{AC}(x,y,1,...,1)$,  $\forall \varphi_1,...,\varphi_n \in [G,G]$. We use that $(x\varphi_1, y\varphi_2)\sim_{AC}(x,y)$ $\forall \varphi_1,\varphi_2\in [G,G]$, and conclude with the following:  $(x\varphi_1, y\varphi_2,\varphi_3,...,\varphi_n)\sim_{AC} (x,y,\varphi_3,...,\varphi_n)\sim_{AC}(x,y,\varphi_3\cdot (\varphi_3)^{-1},...,\varphi_n\cdot(\varphi_n)^{-1})=(x,y,1,...,1)$. 

If $n>\operatorname{rank}(G)>2$ then one concludes that $\Delta_n(G)$ is connected with the same type of arguments. \end{proof}

\medskip
There is a different proof for the fact that if $\langle x_1,...,x_n\rangle=G$ then $\forall c\in [G,G]$: $(x_1,...,x_n)\sim_{AC} (x_1c,x_2,...,x_n)$ (see \cite[Property 2]{Myas}).

\begin{question}
Is it true that, if $\Delta_n(G)$ is connected for some $n\geq w(G)$, then $\Delta_m(G)$ is connected for every $m>n$?
\label{q}
\end{question}
The following proposition gives a partial answer to Question $\ref{q}$:
\begin{proposition}
Let $G$ be in $\mathfrak{C}$ and suppose $\Delta_n(G)$ is connected for some $n\geq w(G)$. Then $\Delta_k(G)$ is connected for every $k\geq n+1$.
\label{conn}
\end{proposition}
\begin{proof}
Let $G=\ll x_1,...,x_n\gg$. Since $G$ is in $\mathfrak{C}$ we have $\langle x_1,...,x_n\rangle=G$. 
By assumption $\Delta_n(G)$ is connected for some $n\geq w(G)$. It is sufficient to prove the statement for $k=n+1$. We need to show that for each normally generating $(n+1)$-tuple $(g_1,...,g_{n+1})$ there is a sequence of Andrews-Curtis moves which transforms $(g_1,...,g_{n+1})$ into $(x_1,...,x_n,1)$. Consider the natural projection $\pi: G\rightarrow \operatorname{Ab}(G)$. Notice that $n+1\geq \operatorname{rank}(\operatorname{Ab}(G))+1$ and therefore $\Delta_{n+1}(\operatorname{Ab}(G))$ is connected by Theorem \ref{thm:diakgrah}.  Hence $(\overline{g}_1,...,\overline{g}_{n+1})\sim_{AC}(\overline{x}_1,...,\overline{x}_n,\overline{1})$ in $\operatorname{Ab}(G)$, and consequently $(g_1,...,g_{+1})\sim_{AC}(x_1c_1,...,x_nc_n,c_{n+1})$ for some $c_1,...,c_{n+1}\in [G,G]$. Since $G$ is in $\mathfrak{C}$ we have that $[G,G]\leq \Phi(G)$ and therefore $\langle x_1c_1,...,x_nc_n\rangle=G$. Hence \begin{align*}&(x_1c_1,...,x_nc_n,c_{n+1})\sim_{AC} (x_1c_1,...,x_nc_n,c_{n+1}\cdot c_{n+1}^{-1})=(x_1c_1,...,x_nc_n,1)\sim_{AC}\\ &(x_1c_1,...,x_nc_n,c_{1})\sim_{AC}(x_1,x_2c_2...,x_nc_n,1)\sim_{AC}...\sim_{AC}(x_1,...,x_n,1). \text{ } \end{align*} \end{proof}

\begin{proof}[Proof of Corollary \ref{ACGrig}] Since $G/[G,G]\cong (\mathbb{Z}/2\mathbb{Z})^3$ (see, for example, \cite{dlHar}, VIII.22),  the graph $\Delta_3(G/[G,G])=\Delta_3((\mathbb{Z}/2\mathbb{Z})^3)=\Gamma_3((\mathbb{Z}/2\mathbb{Z})^3)$ is connected for $n\geq 3$ by Theorem \ref{thm:diakgrah}. To obtain connectedness of $\Delta_n(G)$, $n\geq 3$, apply Theorem \ref{GACC} and Proposition \ref{conn}. \end{proof}

\section{Nielsen equivalence}\label{N}

The question about transitivity of the action of $\operatorname{Aut}F_n$ on $\operatorname{Epi}(F_n,G)$ was raised in 1951 by B.H. Neumann and H. Neumann \cite{NN} and has been studied in different contexts, see \cite{Ev06,Lub,Pak}.  If $G$ is a fundamental group of a closed surface $\Sigma$ then the action of  $\operatorname{Aut}F_n$ is transitive already for $n=\operatorname{rank}(G)$ \cite{Loud, Zies}. There are examples of finitely generated groups when it is not. For instance, for finitely generated abelian groups (see Theorem \ref{thm:diakgrah}) with $m_1\geq 3$. But once $n\geq \operatorname{rank}(G)+1$ transitivity is easier to get. It holds for finitely generated nilpotent groups \cite{Evans}, for finite solvable groups \cite{Dun70}, as well as for $PSL(2,q)$ for $n\geq 4$ and $q$ a prime power \cite{Gar}. 

\begin{proposition}\cite{Evans}
Let $G$ be a  finitely generated group and $n\geq \operatorname{rank}(G)+1$. If $\Gamma_n(G/\Phi(G))$ is connected then $\Gamma_n(G)$ is connected.
\label{evans}
\end{proposition}
\begin{proof} Let $\operatorname{rank}(G)=d$ and
 $G=\langle x_1,...,x_d \rangle$. Assume that $\Gamma_n(G/\Phi(G))$ is connected. Then any generating system $(g_1,...,g_n)$ of $G$ is Nielsen equivalent to $(x_1\varphi_1,...,x_d\varphi_d, \varphi_{d+1},...,\varphi_n)$ for some $\varphi_1,...,\varphi_n \in \Phi(G)$. Since $\varphi_1,...,\varphi_d \in \Phi(G)$ then $\langle x_1\varphi_1,...,x_d\varphi_d\rangle =G$, and we conclude the proof with the following:
\begin{center} $(g_1,...,g_n)\sim (x_1\varphi_1,...,x_d\varphi_d, \varphi_{d+1},...,\varphi_n)\sim (x_1\varphi_1,...,x_d\varphi_d, 1,...,1)\sim (x_1\varphi_1,...,x_d\varphi_d,\varphi_1,1...,1)\sim  (x_1,x_2\varphi_2...,x_d\varphi_d,1...,1)\sim (x_1,x_2,...,x_d,1...,1)$.\end{center}  \end{proof}

In fact, for groups in the class $\mathfrak{C}$ the converse is also true by Proposition \ref{lemma:grig}:

\begin{corollary}
Let $G$ be in $\mathfrak{C}$ and $n\geq\operatorname{rank}(G)+1$. Then $\Gamma_n(G)$ is connected if and only if $\Gamma_n(G/\Phi(G))$ is connected.
\label{con}
\end{corollary}

\begin{corollary}
Let $G$ be in $\mathfrak{C}$ and suppose $\Gamma_n(G)$ is connected for some $n\geq \operatorname{rank}(G)$. Then $\Gamma_k(G)$ is connected for all $k\geq n+1$.
\end{corollary}
\begin{proof} Since $G$ is in $\mathfrak{C}$ then $[G,G]\leq \Phi(G)$. Use that $G/\Phi(G)$ is abelian and apply Theorem \ref{thm:diakgrah} and Corollary \ref{con}. \end{proof}

\subsection{Nielsen equivalence for $p$-groups in class $\mathfrak{C}$}

In this chapter we will prove Theorem \ref{thm1.3}. In view of the Corollary \ref{con} it is desirable to analyse the quotient of the group by its Frattini subgroup.
\begin{proposition}
Let $G$ be a finitely generated $p$-group from $\mathfrak{C}$. Then \begin{center} $G/\Phi(G)\cong (\mathbb{Z}/p\mathbb{Z})^{\operatorname{rank}(G)}.$\end{center}
\label{prop:frat}
\end{proposition}

\begin{proof} First we prove that $\Phi(G)$ is of finite index in $G$ and then show that $G/\Phi(G)\cong (\mathbb{Z}/p\mathbb{Z})^{\operatorname{rank}(G)}.$ 

Let $M$ be a proper maximal subgroup of G. Then by assumption $M\lhd G$ and $G/M$ is an abelian $p$-group, hence finite. Let us prove that $[G:M]=p$. Assume that $[G:M]\neq p$. Since $G$ is a $p$-group it must be $p^k, k\geq 2$. Then $g^{p^k} \in M$ for any $g\in G$. We choose $g \notin M$. There is $0\leq r<k$ such that $g^{p^{r+1}}\in M$ and $g^{p^r}\notin M$. Consider the subgroup $M_1=\langle g, M \rangle$. Obviously $M<M_1$ and $M_1\neq G$. Indeed, assuming that $M_1=G$ we get by the third isomorphism theorem that $|\{e\}|=\frac{|G/M|}{|M_1/M|}=\frac{p^k}{p^r}=p^{k-r}\neq 1$. So $M<M_1<G$ which contradicts the maximality of $M$. Hence $G/M$ has order $p$. 

Since $G$ is finitely generated, there are only finitely many subgroups of a given index \cite{Hall}. Hence there are finitely many maximal subgroups of $G$. We use that the intersection of finitely many subgroups of finite index is a subgroup of finite index \cite{Hall} and conclude that the index of the Frattini subgroup $\Phi(G)$ is finite.  

Since $G$ is in $\mathfrak{C}$, $[G,G]\leq \Phi(G)$ by Proposition \ref{obs}. Also for given $g \in G$ we have $g^p \in M$ for any maximal subgroup M and thus $g^p \in \Phi(G)$. Moreover $G/\Phi(G)$ is finite and we conclude that $G/\Phi(G)$ is an elementary abelian group. By the fundamental theorem of abelian groups $G/\Phi(G)$ is a direct product of nontrivial cyclic subgroups, each of which must have order $p$: $$G/\Phi(G)\cong (\mathbb{Z}/p\mathbb{Z})^{\ell},$$ where $\ell\leq d=\operatorname{rank}(G)$. Suppose that $\ell<d$ and there is a generating $\ell$-tuple $(s_1,...,s_{\ell})$ of $G/\Phi(G)$. By Proposition \ref{lemma:grig} $(s_1,...,s_{\ell})$ has preimage in $\operatorname{Epi}(F_{\ell},G)$ therefore $\operatorname{rank}(G)<d$. We obtain a contradiction. \end{proof}

\begin{proof}[Proof of Theorem \ref{thm1.3}] Use Proposition \ref{prop:frat}, Theorem \ref{thm:diakgrah} and Corollary \ref{con}. \end{proof}
\begin{proof}[Proof of Corollary \ref{gs}] Since $G/\Phi(G)\cong (\mathbb{Z}/p\mathbb{Z})^2$ \cite{Per5}, the graph $\Gamma_2(G/\Phi(G))$ has $\frac{p-1}{2}$ connected components by Theorem \ref{thm:diakgrah}. If two generating tuples of $G$ are Nielsen equivalent in $G$ then their images in $G/\Phi(G)$ are also Nielsen equivalent. We deduce that there are at least $\frac{p-1}{2}\geq 2$ connected components of $\Gamma_2(G)$. \end{proof}

\subsection{Nielsen equivalence of free nilpotent and Heisenberg groups}\label{NilHeis}

Let us recall the following definitions:
\begin{itemize} 
\item A subgroup $K$ of a group $G$ is called \emph{characteristic} if $K$ is invariant by all automorphisms of $G$;
\item Let $W_{\mu}(x_{\lambda})$, $\mu=1,2,...$, be a set of words in the symbols $x_{\lambda}$, $ \lambda=1,2,...$. Then the $\{W_{\mu}\}$-\emph{verbal subgroup} $G(W_{\mu})$ of a group $G$ is the subgroup of $G$ generated by all elements of the form $W_{\mu}(g_{\lambda})$, where  $g_{\lambda}$ ranges over $G$.
\end{itemize}
Any verbal subgroup is characteristic, but the converse is false.


Let $F_n$ be the free group of rank $n\geq 2$ and let $V$ be a verbal subgroup of $F_n$. Then, in particular, $V$ is characteristic and the natural mapping $F_n\rightarrow F_n/V$ induces a homomorphism $\rho: \operatorname{Aut}(F_n)\rightarrow \operatorname{Aut}(F_n/V)$. The elements of the image of $\rho$ are called \emph{tame automorphisms} of $F_n/V$. Since $V$ is a verbal subgroup, a map between two generating $n$-tuples of $F_n/V$ can be extended to an automorphism of $F_n/V$ (see \cite{MaKS}, Theorem 1.1). Thereby the transitivity of the action of $\operatorname{Aut}F_n$ on the set of generating $n$-tuples is reduced to the question whether all automorphisms of $F_n/V$ are tame.

\medskip
Recall that for a finitely generated nilpotent group $G$, $\Gamma_n(G)$ is connected for $n\geq \operatorname{rank}(G)+1$ \cite{Evans}. We will discuss $\Gamma_n(G)$, $n=\operatorname{rank}(G)$ for free nilpotent groups and Heisengroup groups.

Consider the lower central series of $F_n$: $F_n \rhd [F_n,F_n]\rhd [F_n,[F_n,F_n]]...$, every member of which is a verbal subgroup of $F_n$. Let $\gamma_{c}F_n$ be its $c$-th term. The group $F_n(c)=F_n/\gamma_{c+1}F_n$ is the free nilpotent group of rank $n$ and class $c$.

\begin{theorem} [\cite{Andreadakis,Bach}] 
 Let $F_n(c)$ be a free nilpotent group of rank $n\geq 2$ and class $c$. Then 

- all elements of $\operatorname{Aut}F_n(c)$ are tame for $c=1$ or $c=2$;

- if $c\geq 3$ then there exist non-tame automorphisms in $\operatorname{Aut}F_n(c)$.
\label{thm:and}
\end{theorem}

The Theorem \ref{thm1.4} is a straighforward corollary from the Theorem \ref{thm:and}. \\Moreover in the case $n=2$, $c=3$, the result of Andreadakis implies that there are infinitely many Nielsen equivalence classes (Proposition \ref{and} below).
\begin{proposition}
For the free nilpotent group $F_2(3)$ there are infinitely many Nielsen equivalence classes.
\label{and}
\end{proposition}
\begin{proof} Andreadakis \cite{Andreadakis} showed that for $F_2(3)=\langle x,y\rangle$, the central automorphism of $F_2(3)$ \begin{equation}\label{autoand}   \left\{  
           \begin{array}{rcl}  
            x&\rightarrow& x[y,x,x]^{\lambda_1}[y,x,y]^{\lambda_2} \\  
         y&\rightarrow& y[y,x,x]^{\mu_1}[y,x,y]^{\mu_2} \\  
           \end{array}   
           \right . \end{equation} is tame if $\lambda_1=\mu_2$ and $\lambda_2=\mu_1=0$.
We consider a non-tame automorphism of $F_2(3)$: \begin{equation*}\alpha: \left\{  
           \begin{array}{rcl}  
            x&\rightarrow& x[y,x,x]\\ 
            y&\rightarrow& y
           \end{array}   
           \right .,\end{equation*}
and show that there is no $\sigma \in AutF_2$ such that $\alpha^i=\sigma \alpha^j$, for $i>j$. Equivalently, there is no  $\sigma \in AutF_2$ such that $\alpha^k=\sigma$ for an arbitrary positive integer $k$. By the criterion stated above the automorphism $\alpha^k$ is tame if it is trivial. Consider $\alpha^2(x)=x[y,x,x][y,x[y,x,x],x[y,x,x]]=x[y,x,x]^2$. More generally, $$\alpha^k: \left\{  
           \begin{array}{rcl}  
            x&\rightarrow& x[y,x,x]^k\\ 
            y&\rightarrow& y
           \end{array}   
           \right .$$ Hence $\alpha^k$ is not trivial for any $k>0$. \end{proof}

It is easy to see that there is just one Nielsen equivalence class if $c=1$ not using Theorem \ref{thm1.4}. $F_n(1)\cong \mathbb{Z}^n$ and $\operatorname{Epi}(F_n, \mathbb{Z}^n)=GL_n(\mathbb{Z})$. Moreover $\rho(\operatorname{Aut}F_n)=GL_n(\mathbb{Z})$ and the action of $GL_n(\mathbb{Z})$ is obviously transitive on itself. For $n=c=2$ the nilpotent group $F_2(2)$ is isomorphic to the Heisenberg group $\mathcal{H}_1=\langle x,y| [x,[x,y]],[y,[x,y]] \rangle$ \cite{Kahn}. 
\begin{corollary}
For the first Heisenberg group $\mathcal{H}_1$ the graph $\Gamma_2(\mathcal{H}_1)$ is connected.
\end{corollary} 

This statement about the Heisenberg group $\mathcal{H}_1$ can be generalised to higher dimensions. \emph{The discrete Heisenberg group} $\mathcal{H}_k$, $k\geq 1$, is the group of integer matrices of the form \[v=\left(\begin{array}{ccc}
1&\bold{x}&z \\
0&I_k&\bold{y} \\
0&0&1 \end{array} \right)\]  with $\bold{x}=(x_1,...,x_k)$ a row vector of length $k$, $\bold{y}=(y_1,...,y_k)^T$ a column vector of length $k$, and $I_k$ is the $k\times k$ identity matrix.

\bigskip
\begin{proof}[Proof of Theorem \ref{thm1.5}]  
We will prove that $N_{\operatorname{rank}(\mathcal{H}_k)}(\mathcal{H}_k/\Phi(\mathcal{H}_k))$ is connected and deduce connectedness for $N_{\operatorname{rank}(\mathcal{H}_k)}(\mathcal{H}_k)$. We refer to \cite{Evans} for the case $n\geq \operatorname{rank}(\mathcal{H}_k)+1$.

Each element of $\mathcal{H}_k$ can be written as $(x_1,...,x_k, y_1,...,y_k,z)$. The identity element of $\mathcal{H}$ is $(0,0,...,0)$ and $$(x_1,...,x_k, y_1,...,y_k, z)^{-1}=(-x_1,...,-x_k,-y_1,...,-y_k, x_1 y_1+...+x_k y_k-z).$$ The group multiplication is then given by the following rule: \begin{center} $(x_1,...,x_k,y_1,...,y_k,z)(x_1 ',...,x_k ',y_1 ',...,y_k ',z')=(x_1+x_1 ',...,x_k+x_k ',y_1+y_1 ',...,y_k+y_k ',z+z'+x_1 y_1 '+...+x_k y_k ').$ \end{center}

Calculations show that $[\mathcal{H}_k,\mathcal{H}_k]=Z(\mathcal{H}_k)\cong \mathbb{Z}$ and $\operatorname{Ab}(\mathcal{H}_k)=\mathbb{Z}^{2k}$.

Observe that $\mathcal{H}_k=\langle (1,0,...,0,0), (0,1,...,0,0),...,(0,...,1,0) \rangle$ and  \\$\operatorname{rank}({\mathcal{H}_k})=2k.$ Let $\pi: \mathcal{H}_k\rightarrow \operatorname{Ab}(\mathcal{H}_k)$ be the natural epimorphism. Then more generally \begin{center} $\langle h_1,...,h_{2k}\rangle = \mathcal{H}_k$ if and only if  $\langle \pi(h_1),...,\pi(h_{2k})\rangle =\mathbb{Z}^{2k}$.\end{center} Indeed, if $\langle h_1,...,h_{2k}\rangle = \mathcal{H}_k$ then clearly $\langle \pi(h_1),...,\pi(h_{2k})\rangle =\mathbb{Z}^{2k}$. The proof of the inverse implication goes as follows. Let $\langle \pi(h_1),...,\pi(h_{2k})\rangle =\mathbb{Z}^{2k}$. Then $(\pi(h_1),...,\pi(h_{2k}))$ is Nielsen equivalent to $((\overline{1},\overline{0},...,\overline{0},\overline{0}), (\overline{0},\overline{1},...,\overline{0},\overline{0}),...,(\overline{0},...,\overline{1},\overline{0}))$ in $\mathbb{Z}^{2k}$. In other words, $(h_1,...,h_{2k})$ is Nielsen equivalent to \begin{center}$ ((1,0,...,0,m_1), (0,1,...,0,m_2),...,(0,...,1,m_{2k})$.\end{center} The last step is to show that $\langle (1,0,...,0,m_1), (0,1,...,0,m_2),...,(0,...,1,m_{2k})\rangle=\mathcal{H}_k$. We calculate $[(1,0,...,0,m_1),(0,...,1,...0,m_{k+1})]=(0,...,0,1)$. Moreover for all $s_i\in \mathbb{Z}$ letting $a=s_{2k+1}-s_1m_1-...-s_{2k}m_{2k}-s_1s_{k+1}-...-s_k s_{2k}$ we have \begin{center}
$(s_1,s_2,...,s_{2k}, s_{2k+1})=$\\$
[(1,0,...,0,m_1),(0,...,1,...0,m_{k+1})]^{a}\cdot(1,0,...,0,m_1)^{s_1} \cdot...\cdot(0,...,0,1,m_{2k})^{s_{2k}}$. \end{center} We get then that $(h_1,...,h_{2k})$ is Nielsen equivalent to a generating system of $\mathcal{H}_k$ and conclude that $(h_1,...,h_{2k})$ is also a generating system.

Since $\mathcal{H}_k$ is nilpotent group we have
$[\mathcal{H}_k,\mathcal{H}_k]\leq \Phi(\mathcal{H}_k)$ \cite{Evans}. Let us show that $\Phi(\mathcal{H}_k)\leq [\mathcal{H}_k,\mathcal{H}_k]$ or, equivalently, if $g\notin [\mathcal{H}_k,\mathcal{H}_k]$ then $g\notin \Phi(\mathcal{H}_k)$. Assume $g \notin [\mathcal{H}_k,\mathcal{H}_k]$ then $g^{ab}\neq 0$. We deduce that $\exists g_2',...,g_{2k}'\in \mathbb{Z}^{2k}$ such that $\langle g^{ab}, g_2',...,g_{2k}'\rangle=\mathbb{Z}^{2k}$. Therefore $g \notin \Phi(\mathcal{H}_k)$ and $[\mathcal{H}_k,\mathcal{H}_k]=\Phi({\mathcal{H}_k})$.

Notice that $\operatorname{Aut}F_{2k}$ acts transitively on $\mathcal{H}_k/\Phi({\mathcal{H}_k})\cong \mathbb{Z}^{2k}$. Therefore any generating $2k$-tuple of $\mathcal{H}_k$ is Nielsen equivalent to \begin{equation*}((1,0,...,0,m_1), (0,1,...,0,m_2),...,(0,...,1,m_{2k}))\end{equation*} for some integers $m_1,...,m_{2k}$. To obtain connectedness of $N_{2k}(\mathcal{H}_k)$ it is sufficient to prove that there exists $\sigma \in \operatorname{Aut}F_{2k}$ which transforms \begin{equation*}((1,0,...,0,0), (0,1,...,0,0),...,(0,...,1,0))\end{equation*} into \begin{equation*}((1,0,...,0,m_1), (0,1,...,0,m_2),...,(0,...,1,m_{2k}))\end{equation*} for any $m_i \in \mathbb{Z}$.

One calculates \begin{center}$(L^{-1}_{1,k+1}R_{1,k+1})^{m_1}((1,0,...,0,0), (0,1,...,0,0),...,(0,...,1,0))=((1,0,...,0,m_1), (0,1,...,0,0),...,(0,...,1,0)).$\end{center} 
More generally $\forall i: 1\leq i\leq k$ $$(L^{-1}_{i,k+i}R_{i,k+i})^{m_i}(0,...,1,...,0,0)=(0,...,1,...,0,m_i)$$ and $\forall k: k+1\leq i \leq 2k$
$$(L_{i,i-k}R^{-1}_{i,i-k})^{m_i}(0,...,1,...,0,0)=(0,...,1,...,0,m_i). $$ \end{proof}

\bigskip
\emph{Section de Math\'ematiques, Universit\'e de Gen\`eve \\2-4 rue du Li\`evre, 1211 Gen\`eve, Switzerland
\\email: Aglaia.Myropolska@unige.ch}


\begin{thebibliography}{9}
\bibitem{AkKir} S.\ Akbulut, R.\ Kirby, \emph{A potential smooth counterexample in dimension $4$ to the Poincare conjecture, the Schoenflies conjecture, and the Andrews-Curtis conjecture}, Topology \textbf{24} (4) (1985), 375--390.
\bibitem{Andreadakis} S.\ Andreadakis, \emph{On the automorphisms of free groups and free nilpotent groups}, Proc.\ London Math.\ Soc.\ (3) \textbf{15} (1965), 239--268.
\bibitem{AC} J.\ J.\ Andrews, M.\ L.\ Curtis, \emph{Free groups and handlebodies}, Proc.\ Amer.\ Math.\ Soc.\ \textbf{16} (1965), 192--195.
\bibitem{Bach} S.\ Bachmuth, \emph{Induced automorphisms of free groups and free metabelian groups}, Trans.\ Amer.\ Math.\ Soc.\ \textbf{122} (1966), 1--17.
\bibitem{BorKhukMyas} A.\ V.\ Borovik, E.\ I.\ Khukhro, A.\ G.\ Myasnikov, \emph{The Andrews-Curtis Conjecture and black box groups}, Internat.\ J.\ Algebra Comput.\ \textbf{13} (4) (2003), 415--436.
\bibitem{BorLubMyas} A.\ V.\ Borovik, A.\ Lubotzky, A.\ G.\ Myasnikov, \emph{The Finitary Andrews-Curtis Conjecture}, Infinite groups: geometric, combinatorial and dynamical aspects, pp.\ 15--30, Progr.\ Math., \textbf{248}, Birkh\"auser, Basel, 2005.
\bibitem{Bour} N.\ Bourbaki, \emph{Alg\'ebre} 1--3, Springer, Berlin, 2006.
\bibitem{DiacGrah} P.\ Diaconis, R.\ Graham, \emph{The graph of generating sets of an abelian group}, Colloq.\ Math.\ \textbf{80} (1999), 31--38.
\bibitem{Dun63} M.\ J.\ Dunwoody, \emph{On T-systems of groups}, J. Austral.\ Math.\ Soc.\ \textbf{3} (1963), 172--179.
\bibitem{Dun70} M.\ J.\ Dunwoody, \emph{Nielsen transformations}, Computational Problems in Abstract Algebra, pp.\ 45--46, Pergamon, Oxford, 1970.
\bibitem{Evans} M.\ J.\ Evans, \emph{Presentations of group involving more generators than are necessary}, Proc.\ London Math.\ Soc.\ \textbf{67} (1) (1993), 106--126.
\bibitem{Ev06} M.\ J.\ Evans, \emph{Nielsen equivalence classes and stability graphs of finitely generated groups}, Ischia Group Theory 2006, pp.\ 103--119, World Scientific.
\bibitem{Gar} S.\ Garion, \emph{Connectivity of the Product Replacement Algorithm Graph of $PSL(2,q)$}, J.\ Group Theory \textbf{11} (6) (2008), 765--777.
\bibitem{Grig80} R.\ I.\ Grigorchuk, \emph{On Burnside's problem on periodic groups}, Funktsional. Anal. i Prilozhen. \textbf{14} (1) (1980), 53--54.
\bibitem{Gr84} R.\ I.\ Grigorchuk, \emph{Degrees of growth of finitely generated groups and the theory of invariant means}, Izv.\ Akad.\ Nauk SSSR Ser.\ Mat.\ \textbf{48} (5) (1984), 939--985.
\bibitem{GrigSelecta} R.\ I.\ Grigorchuk, \emph{Solved and unsolved problem around one group}, Infinite Groups: Geometric, Combinatorial and Dynamical Aspects, pp.\ 117--218,
Progr.\ Math., 248, Birkh\"auser, Basel, 2005. 
\bibitem{Gupta-Sidki} N.\ Gupta, S.\ Sidki, \emph{Some infinite $p$-groups}, Algebra i Logika \textbf{22} (1983), 584--589.
\bibitem{Hall} M.\ Hall, \emph{A topology for free groups and related topics}, Annals Math.\ \textbf{52} (1950), 127--139.
\bibitem{dlHar} P.\ de la Harpe, \emph{Topics in geometric group theory}, Chicago Lectures in Mathematics, University of Chicago Press, Chicago, IL, 2000.
\bibitem{Kahn} P.\ J.\ Kahn, \emph{Automorphisms of the discrete Heisenberg group}, http://www.math.cornell.edu/People/Faculty/Heisen.pdf, 2005.
\bibitem{Kerv} M.\ A.\ Kervaire, \emph{Les noeuds de dimensions sup\'erieures}, Bull.\ Soc.\ Math.\ France \textbf{93} (1965), 225--271.
\bibitem{Loud} L.\ Louder, \emph{Nielsen equivalence of generating sets for closed surface groups} (submitted) http://arxiv.org/abs/1009.0454, 2010.
\bibitem{Lub} A.\ Lubotzky, \emph{Dynamics of $\operatorname{Aut}(Fn)$ actions on group presentations and representations}, Geometry, rigidity, and group actions, pp.\ 609--643, 
Chicago Lectures in Math., Univ.\ Chicago Press, Chicago, IL, 2011.
\bibitem{LS} R.\ C.\ Lyndon, P.\ E.\ Schupp, \emph{Combinatorial group theory}, Springer-Verlag, Berlin (1977), Ergebnisse der Mathematik und ihrer Grenzgebiete, Band 89.
\bibitem{MaKS} W.\ Magnus, A.\ Karrass, D.\ Solitar, \emph{Combinatorial Group Theory}, Interscience Publishers, New York - London - Sydney, 1966.
\bibitem{MalPak} A.\ Malyshev, I.\ Pak, \emph{Growth in Product Replacement graphs}, arXiv:1304.5320, 2013.
\bibitem{Myas} A.\ G.\ Myasnikov, \emph{Extended Nielsen transformations and the trivial group}, Mat.\ Zametki \textbf{35} (4) (1984), 491--495.
\bibitem{Neum} B.\ H.\ Neumann, \emph{Some remarks on infinite groups}, Proc.\ London Math.\ Soc.\ \textbf{12} (1937), 120--127.
\bibitem{NN} B.\ H.\ Neumann, H.\ Neumann, \emph{Zwei Klassen charakteristischer Untergruppen und ihre Faktorgruppen}, Math.\ Nachr.\ \textbf{4} (1951), 106--125.
\bibitem{Oan} D.\ Oancea, \emph{A note on Nielsen equivalence in finitely generated abelian groups}, Bull.\ Aust.\ Math.\ Soc.\ \textbf{84} (1) (2011), 127--136.
\bibitem{Pak} I.\ Pak, \emph{What do we know about the product replacement algorithm?},  Groups and computation, III, pp.\ 301--347, Ohio State Univ.\ Math.\ Res.\ Inst.\ Publ., 8, de Gruyter, Berlin, 2001.
\bibitem{Pervova} E.\ L.\ Pervova, \emph{Everywhere Dense Subgroups of One Group of Tree Automorphisms}, (Russian) Tr.\ Mat.\ Inst.\ Steklova \textbf{231} (2000), Din.\ Sist., Avtom.\ i Beskon.\ Gruppy, 356--367; translation in Proc.\ Steklov Inst.\ Math.\ 2000, \textbf{4} (231), 339--350.
\bibitem{Per5} E.\ L.\ Pervova, \emph{Maximal subgroups of some non locally finite $p$-groups}, Internat. J.\ Algebra Comput.\ \textbf{15} (5--6) (2005), 1129--1150.
\bibitem{Rob} D.\ J.\ S.\ Robinson, \emph{A Course in the Theory of Groups}, Graduate Texts in Mathematics, 80, Springer-Verlag, New York - Berlin, 1982.
\bibitem{Zies} H.\ Zieschang, \emph{\"Uber die Nielsensche K\"urzungsmethode in freien Produkten mit Amalgam}, Invent.\ Math.\ \textbf{10} (1970), 4--37.

\end{thebibliography}
\end{document}